\documentclass[12pt, a4paper]{amsart}
\pdfoutput=1

\usepackage[T1]{fontenc}
\usepackage[utf8]{inputenc}
\usepackage{a4wide}
\usepackage{hyperref}
\usepackage{mathtools}
\usepackage{tikz}
\usetikzlibrary{arrows,backgrounds,calc}

\newif\ifprivate
\ifprivate
\usepackage[mark]{gitinfo2}
\renewcommand{\gitMark}{\jobname\,\textbullet{}\,\gitFirstTagDescribe\,\textbullet{}\,\gitAuthorName,\,\gitAuthorIsoDate}
\newcommand{\TODO}[1]%
{\par\fbox{\begin{minipage}{0.9\linewidth}\textbf{TODO:} #1\end{minipage}}\par}
\fi

\DeclarePairedDelimiter{\abs}{\lvert}{\rvert}

\newtheorem{theorem}{Theorem}
\newtheorem{lemma}{Lemma}[section]
\newtheorem{proposition}[lemma]{Proposition}

\newcommand{\calR}{\mathcal{R}}
\newcommand{\calT}{\mathcal{T}}
\newcommand{\calV}{\mathcal{V}}
\newcommand{\calZ}{\mathcal{Z}}
\newcommand{\E}{\mathbb{E}}
\renewcommand{\MR}[1]{}
\renewcommand{\P}{\mathbb{P}}
\newcommand{\V}{\mathbb{V}}

\begin{document}

\title{Protection Number in Plane Trees}
\author[C.~Heuberger]{Clemens Heuberger}
\address[Clemens Heuberger]{Institut f\"ur Mathematik, Alpen-Adria-Uni\-ver\-si\-t\"at Klagenfurt,
  Universit\"atsstra\ss e 65--67, 9020 Klagenfurt, Austria}
\email{clemens.heuberger@aau.at}
\thanks{C.~Heuberger is supported by the Austrian Science Fund (FWF):
  P~24644-N26. This paper has been written while he was a visitor at Stellenbosch University.}

\author[H.~Prodinger]{Helmut Prodinger}
\thanks{H.~Prodinger is supported by an incentive grant of the National Research Foundation of South Africa.}

\address[Helmut Prodinger]{Department of Mathematical Sciences, Stellenbosch University, 7602 Stellenbosch,
 South Africa}
\email{hproding@sun.ac.za}


\begin{abstract}
  The protection number of a plane tree is the minimal distance of the root to
  a leaf; this definition carries over to an arbitrary node in a plane tree by
  considering the maximal subtree having this node as a root. We study the the
  protection number of a uniformly chosen random tree of size $n$ and also the
  protection number of a uniformly chosen node in a uniformly chosen random
  tree of size $n$. The method is to apply singularity analysis to appropriate
  generating functions. Additional results are provided as well.
\end{abstract}

\maketitle

\section{Introduction}

Cheon and Shapiro \cite{Cheon-Shapiro:2008:protec} started the study of 2-protected nodes in trees. A node
enjoys this property if its distance to any leaf is at least 2. After this
pioneering paper, a large number of papers has been published \cite{Mansour:2011:protec,
  Du-Prodinger:2012:notes, Bona:2014, Holmgren-Janson:2015:asymp,
  Devroye-Janson:2014:protec, Mahmoud-Ward:2015:asymp}.

In this paper we study the protection number of the root of a (rooted, plane)
tree (in the older literature often called ordered tree): It is the minimal
distance of the root to any leaf.  Further, the protection number of any node
is defined by taking the (maximal) subtree that has this node as the root.

Preliminary results on the subject have appeared in the recent paper
\cite{Copenhaver:2016}, but we show that, thanks to a rigorous use of methods
outlined in the book \emph{Analytic Combinatorics}~\cite{Flajolet-Sedgewick:ta:analy}, we can go much
further. We are able to solve a basic recurrence explicitly, which allows us to
use singularity analysis of generating functions and getting, at least in
principle, as many terms as one wants in the asymptotic expansions of
interest. Further, one can derive explicit expressions for the probabilities in
question.

Some curious observations related to the constants that appear are also made;
they are linked to identities due to Dedekind, Ramanujan and others and are
part of the toolkit of the \emph{modern analysis of algorithms}.

\section{Results}

In a rooted plane tree $t$, a vertex is said to be $k$-protected if its minimum
distance from a leaf is at least $k$. The tree $t$ is said to be $k$-protected
if its root is $k$-protected.

We denote the maximal $k$ such that a
vertex $v$ is $k$-protected by $\pi(v)$ and call it the protection number of
$v$. The protection number of the root is called the protection number of the
tree, $\pi(t)$. This means that a tree $t$ is $k$-protected if and only if
$\pi(t)\ge k$.

If a vertex $v$ is a leaf, then $\pi(v)=0$ by definition. Otherwise, if it has
children $w_1$, \ldots, $w_\ell$, then
\begin{equation*}
  \pi(v) = 1 + \min\{\pi(w_1), \ldots, \pi(w_\ell)\}.
\end{equation*}

We are interested in the random variables
\begin{itemize}
\item $X_n$, the protection number of a uniformly chosen random tree with $n$
  vertices,
\item $Y_{n}$, the protection number of a uniformly chosen vertex in a uniformly chosen
  random tree with $n$ vertices.
\end{itemize}

We will prove the following results.

\begin{theorem}\label{theorem:protection-number-tree}
  For $n\to\infty$, the protection number $X_n$ of a tree with $n$ vertices
  tends to a discrete limit distribution:
  \begin{multline*}
    \P(X_n=k) = \frac{27 \cdot 4^k  (4^{2k} - 1)}{(4^k + 2)^2 (2\cdot
      4^k+ 1)^2}\\
{\textstyle + \frac{81 \cdot 4^k  (4 (k-3) 4^{6k} + 36\cdot 4^{5k} - (45
    k-72) 4^{4k} - 80 k 4^{3k}  - (45 k+72) 4^{2k} - 36\cdot 4^k+ 4 (k+3))}{2
    (4^k + 2)^{4} (2\cdot 4^k + 1)^{4}}} \frac1{n} \\
 +
    O\Bigl(\frac{k^2}{3^k n^{3/2}}\Bigr).
  \end{multline*}

  Setting
 \begin{align*}
   c_0 &= \sum_{k\ge 1} \frac{9\cdot 4^k}{(4^{k} +2)^2} \\
       &= 1.622971384715353049514658203184345989635513668984063539407825\ldots, \\
   c_1 &= \sum_{k\ge 1} \frac{9\cdot 4^k  ((3 k-8) 4^{2k}+ 28\cdot 4^k - (12
         k+20))}{2(4^k + 2)^4} \\
       &=
         0.1311873689494231825244485810366733833577429413531428274982796\ldots,\\
   c_2 &= \sum_{k\ge 1}  \frac{9(2k-1)4^k}{(4^{k} +2)^2}  - c_0^2\\
       &= 0.71569507178333266731548919868273628601066118785422617431075\ldots, \\
   c_3 &=\sum_{k\ge 1}\frac{9(2k-1) 4^k  ((3 k-8) 4^{2k}+ 28\cdot 4^k - (12
         k+20))}{2(4^k + 2)^4} - 2c_0c_1\\
       &=-0.294639322732595323433878185755458143829498855158644070705218\ldots, 
 \end{align*}
 its expectation and variance can be written as
 \begin{align*}
   \E(X_n) &= c_0 + c_1 \frac{1}{n} + O\Bigl(\frac1{n^{3/2}}\Bigr),\\
   \V(X_n) &= c_2 + c_3 \frac{1}{n} + O\Bigl(\frac1{n^{3/2}}\Bigr).
 \end{align*}
\end{theorem}

\begin{theorem}\label{theorem:vertex-protection-number}
  For $n\to \infty$, the protection number $Y_n$ of a random vertex of a random
  tree with $n$ vertices tends to a discrete limit distribution:
  \begin{multline*}
    \P(Y_n=k) = \frac{9\cdot 4^k}{(4\cdot 4^{k} +2)(4^k + 2)} \\
    + \frac{3\cdot 4^k (4^k - 1) ((6k-22)4^{3k} + (21k+30)4^{2k}  + (21k+96)4^k  + (6k + 58))}{(4^k + 2)^3  (2\cdot4^k + 1)^3} \frac1{n} \\
    + O\Bigl(\frac{k^2}{3^kn^{2}}\Bigr).
  \end{multline*}

  Setting
  \begin{align*}
    d_0 &=\sum_{k\ge 1}\frac{3}{4^k + 2}\\
    &=0.727649276913726097531184400482145348863515722775042276537008\ldots,\\
    d_1 &=\sum_{k\ge 1}\frac{(3 k-10) 4^{2 k} + (6 k+26) 4^k -
          16}{2(4^k+2)^3}\\
    &=-0.0311837125986222774945246489936100437425899128713521725307175\ldots,\\
    d_2 &= \sum_{k\ge 1}\frac{3(2k-1)}{4^k + 2} -d_0^2\\
    &=0.81689937948362892278879205623322983539562628691031631640757\ldots,\\
    d_3 &= \sum_{k\ge 1}\frac{(2k-1)((3 k-10) 4^{2 k} + (6 k+26) 4^k -
          16)}{2(4^k+2)^3} - 2d_0d_1\\
    &=0.014197899249123624176745586362758197533680269252844749278840\ldots,
  \end{align*}
  its expectation and variance can be written as
  \begin{align*}
    \E(Y_n) &= d_0 + d_1\frac1n  + O\Bigl(\frac1{n^{3/2}}\Bigr),\\
    \V(Y_n) &= d_2 + d_3\frac1n + O\Bigl(\frac1{n^{3/2}}\Bigr).
  \end{align*}
  
\end{theorem}

\section{\texorpdfstring{Number of $k$-Protected Trees}{Number of k-Protected Trees}}
In this section, we investigate the auxiliary quantity $r_{nk}$, the number of
$k$-protected trees with $n$ vertices.

Let $\calR_{\ge k}$ denote the set of all rooted plane trees
with protection number $\ge k$. Its generating function is denoted
by $R_{\ge k}(z)$ where $z$ labels the number of vertices of a tree, i.e.,
\begin{equation*}
  R_{\ge k} (z)= \sum_{n\ge 1}r_{nk}z^n.
\end{equation*}

\begin{lemma}
  We have
  \begin{equation}\label{eq:R_0-formula}
    R_{\ge 0}(z)=\frac{1-\sqrt{1-4z}}2,
  \end{equation}
  and
  \begin{equation}\label{eq:R_k_formula}
    R_{\ge k}(z) = \frac{(1-z)z^{k-2}(R_{\ge 0}(z))^3}{1+z^{k-2}(R_{\ge 0}(z))^3}
  \end{equation}
  for $k\ge 1$.
\end{lemma}
\begin{proof}
  It is clear that $R_{\ge 0}$ is the generating function of all rooted plane
  trees which is well-known to be given by \eqref{eq:R_0-formula}, cf.\ for
  instance \cite[\S~I.5.1]{Flajolet-Sedgewick:ta:analy}.

  For $k\ge 1$, the root of a tree is $k$-protected if and only if all of its children are
  $(k-1)$-protected. Thus a tree is $k$-protected if and only if it
  consists of a root and a non-empty sequence of branches whose roots are
  $(k-1)$-protected. This translates into the symbolic equation shown in
  Figure~\ref{fig:k-protected-tree-symbolic} and thus into the equation
  \begin{equation}\label{eq:recurrence-R_k}
    R_{\ge k} (z) = \frac{zR_{\ge k-1}(z)}{1-R_{\ge k-1}(z)}
  \end{equation}
  for $k\ge 1$.
  \begin{figure}[ht]
    \centering
    \begin{tikzpicture}
      \node (add) {$\calR_{\ge k}=$};
      \node[right of=add, draw, inner sep=3pt, xshift=7em, yshift=2em] (right-V) {};
      \node[below of=right-V, yshift=-1.5em, xshift=-6em] (1) {$\calR_{\ge k-1}$};
      \node[below of=right-V, yshift=-1.5em, xshift=-3em] (2) {$\calR_{\ge k-1}$};
      \node[below of=right-V, yshift=-1.5em, xshift=0em]  (3) {$\calR_{\ge k-1}$};
      \node[below of=right-V, yshift=-1.5em, xshift=3em, gray] (4) {$\cdots$};
      \node[below of=right-V, yshift=-1.5em, xshift=6em] (5) {$\calR_{\ge k-1}$};
      \draw (right-V) -- (1) (right-V) -- (2) (right-V) -- (3) (right-V) -- (5);
      \draw[dotted, gray] (right-V) -- (4);
    \end{tikzpicture}
    \caption{Symbolic equation for $\calR_{\ge k}$.}
    \label{fig:k-protected-tree-symbolic}
  \end{figure}
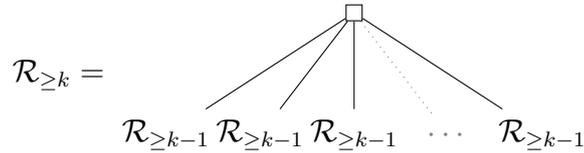

  It would now be easy to prove~\eqref{eq:R_k_formula} by induction; however,
  we intend to \emph{derive}~\eqref{eq:R_k_formula}.

  We rewrite the recurrence~\eqref{eq:recurrence-R_k} in the form
  \begin{equation*}
    R_{\ge k}(z) = -z +\frac{z}{1-R_{\ge k-1}(z)}
  \end{equation*}
  such that $R_{\ge k-1}$
  occurs only once.

  Using the abbreviation $F_k \coloneqq R_{\ge k}(z)+z$ this leads to
  \begin{equation*}
    F_k = \frac{z}{1+z-F_{k-1}}.
  \end{equation*}
  This is reminiscent of continued fractions. We use the Ansatz
  $F_k=z a_k/a_{k+1}$ resulting in the equation
  \begin{equation*}
    \frac{z a_k}{a_{k+1}} = \frac{z a_k}{(1+z)a_k - za_{k-1}}.
  \end{equation*}
  It is sufficient to require
  \begin{equation*}
    a_{k+1} = (1+z)a_k -z a_{k-1}
  \end{equation*}
  for $k\ge 1$.

  This is a linear recurrence whose characteristic equation has roots $1$ and
  $z$, so its solution has the form $a_k= A + Bz^k$. As common factors between
  $a_k$ and $a_{k+1}$ do not matter, we may choose $A=z^3$ which leads to
  $B=(R_{\ge 0}(z))^3$.

  Thus
  \begin{equation*}
    R_{\ge k}(z) = \frac{z(z^3+z^{k}(R_{\ge 0}(z))^3)}{z^3+z^{k+1}(R_{\ge 0}(z))^3} -z 
  \end{equation*}
  which results in~\eqref{eq:R_k_formula}.
\end{proof}

\begin{proposition}\label{proposition:root-k-protected}The probability that a
  tree is $k$ protected is
  \begin{equation}\label{eq:root-k-protected}
  \begin{aligned}
    \P(X_n\ge k)&=\frac{9\cdot 4^k}{(4^{k} +2)^2} \\
                &\qquad+ \frac{9\cdot  4^k  ((3
                  k-8) 4^{2k}+ 28\cdot 4^k
                  - (12 k+20))}{2(4^k +
                  2)^4} \frac1n\\
    &\qquad + O\Bigl(\frac{k^2}{3^k n^{3/2}}\Bigr).  
  \end{aligned}
  \end{equation}
\end{proposition}
\begin{proof}

  We intend to use singularity analysis (\cite{Flajolet-Odlyzko:1990:singul}, \cite[Chapter~VI]{Flajolet-Sedgewick:ta:analy}).
  Let  $z$ be in some $\Delta$-domain at $1/4$ (see \cite[Definition
  VI.1]{Flajolet-Sedgewick:ta:analy}) with $\abs{z-1/4}\le 1/12$.
  We have
  \begin{align*}
    z^{k-2} &= \Bigl(\frac14 - \Bigl(\frac14 - z\Bigr)\Bigr)^{k-2} \\
            &= 
              \frac{16}{4^k} - \frac{16(k-2)}{4^k} (1-4z)+ O\Bigl(\frac{k^2}{3^k}(1-4z)^2\Bigr).
  \end{align*}
  Inserting this into \eqref{eq:R_k_formula}, we get
  \begin{equation}\label{eq:R_k_singular_expansion}
  \begin{aligned}
    R_{\ge k}(z) &= \frac{3}{2\cdot4^k + 4} + \frac{-9\cdot4^k}{2\cdot4^{2k} +
                   8\cdot4^k + 8} (1-4z)^{1/2} \\
&\qquad+ \frac{-3 k\cdot4^{2k} - 6 k 4^k + 16\cdot4^{2k} - 20\cdot4^k +
  4}{2\cdot4^{3k} + 12\cdot4^{2k} + 24\cdot4^k + 16}(1-4z)\\
&\qquad + \frac{9 k 4^{3k} - 24\cdot4^{3k} - 36 k 4^k + 84\cdot4^{2k} -
  60\cdot4^k}{2\cdot4^{4k} + 16\cdot4^{3k} + 48\cdot4^{2k} + 64\cdot4^k +
  32}(1-4z)^{3/2} \\
&\qquad + O\Bigl(\frac{k^2}{3^k}(1-4z)^2\Bigr).
  \end{aligned}
  \end{equation}
  By singularity analysis, we obtain
  \begin{multline*}
    r_{nk} = 
             \Bigl(\frac{9\cdot 4^k}{4 \sqrt{\pi} (4^{2 k} + 4\cdot 4^k + 4)}\Bigr) 4^{n}
             n^{-\frac{3}{2}} \\
           + \Bigl(\frac{9  (12\cdot 4^{3 k} k - 29\cdot 4^{3 k} + 124\cdot 4^{2 k} -
             48\cdot 4^k k - 68\cdot 4^k)}{32 \sqrt{\pi}  (4^{4 k} + 8\cdot
             4^{3 k} + 24\cdot 4^{2 k} + 32\cdot 4^k + 16)}\Bigr) 4^{n}
             n^{-\frac{5}{2}}\\
            + O\Bigl(\frac{k^24^{n}}{3^k n^3}\Bigr)
  \end{multline*}

  Singularity analysis and division by the number $C_{n-1}$ (the $(n-1)$st
  Catalan-number) of rooted plane trees (this corresponds to setting $k=0$
  above) yields~\eqref{eq:root-k-protected}.

\end{proof}

\begin{proof}[Proof of Theorem~\ref{theorem:protection-number-tree}]
  We use $\P(X_n=k)=\P(X_n\ge k) - \P(X_n\ge k+1)$ and
  Proposition~\ref{proposition:root-k-protected} to prove the limit theorem.

  The expectation follows from the well-known formula
  \[ \E(X_n)=\sum_{k\ge 1}\P(X_n\ge k) \]
  which is valid for all random variables with non-negative integer values.

  Similarly, the variance follows from $\V(X_n)=\E(X_n^2)-\E(X_n)^2$ and
  \begin{align*}
    \E(X_n^2)&=\sum_{k\ge 1} k^2 \P(X_n=k)
    = \sum_{k\ge 1}(2k-1)\P(X_n\ge k).
  \end{align*}
\end{proof}

\section{Protection Numbers of all Vertices}

We now turn to the protection numbers of arbitrary vertices. We fix some $k$
and consider the number $s_{nk}$ of $k$-protected vertices summed over all
trees of size $n$. The corresponding generating function is denoted by
\begin{equation*}
  S_{\ge k}(z) = \sum_{n\ge 0} s_{nk} z^n
\end{equation*}
where $z$ again labels the number of vertices.

\begin{lemma}We have
  \begin{equation}\label{eq:S_k-formula}
    S_{\ge k}(z) = \frac12 R_{\ge k}(z) \Bigl( 1 + \frac{1}{\sqrt{1-4 z}}\Bigr).
  \end{equation}
\end{lemma}
\begin{proof}

In the language of the symbolic method, the generating function $S_{\ge k}(z)$ corresponds to
the class $\Theta_{\calZ_{\ge k}} \calT$ where $\Theta_{\ge \calZ_k}$ denotes
\emph{pointing} to a $k$-protected vertex,
cf.~\cite[Definition~I.14]{Flajolet-Sedgewick:ta:analy}.

A tree $t$ and a $k$-protected vertex $w$ in this tree bijectively correspond to
a $k$-protected tree $t_1$ whose root is merged with a leaf of another tree $t_2$, cf.\
Figure~\ref{figure:decompose-at-k-protected-vertex}.
  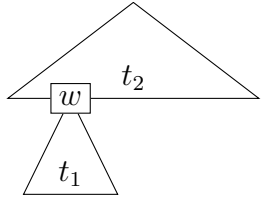
\begin{figure}[ht]
    \centering
    \begin{tikzpicture}
      \coordinate (root);
      \coordinate[below of=root, yshift=-1.5ex, xshift=-4em] (1);
      \coordinate[below of=root, yshift=-1.5ex, xshift=4em] (5);
      \node[below of=root, draw, inner sep=3pt, yshift=-1.5ex, xshift=-2em] (w)
      {$w$};
      \coordinate[below of=w, yshift=-1.5ex, xshift=-1.5em] (3);
      \coordinate[below of=w, yshift=-1.5ex, xshift=1.5em] (4);
      \draw (root) -- (1) -- (w) -- (3) -- (4) -- (w)-- (5) --(root);
      \node[below of=root] {$t_2$};
      \node[below of=w] {$t_1$};
    \end{tikzpicture}
    \caption{Decomposition at a $k$-protected vertex.}
    \label{figure:decompose-at-k-protected-vertex}
  \end{figure}
Thus $\calZ \times \Theta_{\ge k}\calT$ corresponds bijectively to $\calR_{\ge k} \times \Theta_\calV \calT$ where $\Theta_\calV$ denotes
pointing at a leaf and the factor $\calZ$ on the left hand side denotes one
single vertex which compensates the fact that merging the root of one tree with a leaf
of the other tree reduces the number of vertices by $1$.

Thus $S_{\ge k}(z) = z^{-1} R_{\ge k}(z) L(z)$ where $L(z)$ denotes the
generating function of $\Theta_{\calV} \calT$ with respect to the number of
vertices. Note that the pointing is with respect to leaves, but $L$ is a
generating function with respect to all vertices.

Let
\begin{equation*}
  T(v, z) = \sum_{n\ge 0}\sum_{\ell\ge 0} N_{n-1, \ell} v^\ell z^n
\end{equation*}
denote the generating function of $\calT$ where $z$ marks the number of
vertices and $v$ marks the number of leaves. Here, $N_{n-1, \ell}$ denotes the
Narayana number counting the number of trees with $n$ vertices and $\ell$ leaves.

It is a well-known consequence of the symbolic method that
\begin{equation*}
  T(v, z) = zv + \frac{zT(v, z)}{1-T(v, z)},
\end{equation*}
cf.\ \cite[Example~III.13]{Flajolet-Sedgewick:ta:analy}. This yields the
explicit expression 
\begin{equation*}
  T(v, z) = \frac{1-  z + v z  -\sqrt{(v-1)^2 z^{2} - 2  (v + 1) z + 1}}{2}. 
\end{equation*}
Pointing corresponds to applying
$v\frac{d}{dv}$, cf.\ \cite[Theorem~I.4]{Flajolet-Sedgewick:ta:analy}. Setting
$v=1$ then leads to
\begin{equation*}
  L(z) = \Bigl(v\frac{dT(v, z)}{dv}\Bigr)\Bigr|_{v=1} = \frac{z}{2} \Bigl( 1 + \frac{1}{\sqrt{1-4 z}}\Bigr).
\end{equation*}
This yields \eqref{eq:S_k-formula}.
\end{proof}

\begin{proposition}\label{proposition:vertex-k-protected}
  The probability that a random vertex of a random tree with
  $n$ vertices has protection number at least $k$ is
  \begin{equation*}\label{eq:expectation-k-protected-vertices}
    \P(Y_{n}\ge k) = \frac{3}{4^k + 2}
    + \frac{(3 k-10)4^{2 k} + (6 k +26) 4^k - 16}{2(4^k+2)^3} \frac{1}{n} + O\Bigl(\frac{k^2}{3^kn^2}\Bigr).
  \end{equation*}
\end{proposition}
\begin{proof}
  By \eqref{eq:S_k-formula} and \eqref{eq:R_k_singular_expansion}, we get
  \begin{align*}
    S_{\ge k}(z) &=  \frac{3}{4\cdot 4^k + 8} (1-4z)^{1/2} - \frac{3\cdot 4^k -
      3}{2\cdot 4^{2 k} + 8\cdot 4^k + 8}\\ &\qquad-\frac{(3 k-7) 4^{2 k} + (6 k +38) 4^k  - 4}{4\cdot 4^{3 k} + 24\cdot 4^{2 k} + 48\cdot 4^k
      + 32} (1-4z)^{-1/2} \\&\qquad+
    \frac{(3 k-4) 4^{3 k} - 6 (k -8) 4^{2 k} -
      24 (k+2) 4^k  + 4}{2\cdot 4^{4 k} + 16\cdot
      4^{3 k} + 48\cdot 4^{2 k} + 64\cdot 4^k + 32} (1-4z)^{-1}\\
    &\qquad+ O\Bigl(\frac{k^2}{3^k(1-4z)^{-3/2}}\Bigr).
  \end{align*}
  By singularity analysis, we get
  \begin{multline*}
    s_{nk} = \frac{3}{4{\sqrt{\pi}(4^k + 2)}} 4^{n}
    n^{-1/2} \\+ \frac{(12k-31) 4^{2 k} + (24k+140) 4^k - 28}{32
      \sqrt{\pi} {(4^{3 k} + 6 \cdot 4^{2 k} + 12 \cdot 4^k + 8)}} 4^{n}
    n^{-3/2} + O\Bigl(\frac{k^2}{3^kn^{5/2}}\Bigr).
  \end{multline*}

  Dividing by the number $C_{n-1}$ of all trees and by the number $n$ of vertices
  yields~\eqref{eq:expectation-k-protected-vertices}.
\end{proof}

\begin{proof}[Proof of Theorem~\ref{theorem:vertex-protection-number}]
  Theorem~\ref{theorem:vertex-protection-number} follows from
  Proposition~\ref{proposition:vertex-k-protected} in the same way as
  Theorem~\ref{theorem:protection-number-tree} follows from
  Proposition~\ref{proposition:root-k-protected}.
\end{proof}

\section{Explicit formula for the number of \texorpdfstring{$\ge k$}{≥ k}-protected trees}

Our goal is to read off the coefficient of $z^n$ in
formula~\eqref{eq:R_k_formula} in explicit form.

\begin{proposition}The number of $k$-protected trees with $n$ vertices is
  \begin{equation*}
    r_{nk}=\sum_{j\ge1}(-1)^{j-1}\biggl[\binom{2n-3-(2k-1)j}{n-(k+1)j}-\binom{2n-3-(2k-1)j}{n-3-(k+1)j}\biggr].
\end{equation*}
\end{proposition}
\begin{proof}We use the substitution
$z=u/(1+u)^2$, which was introduced in \cite{Bruijn-Knuth-Rice:1972} and rewrite~\eqref{eq:R_k_formula} as
\begin{align*}
G_{\ge k}\left(\frac{u}{(1+u)^2}\right)&=\frac{1-u^3}{(1-u)(1+u)^2} \frac{\frac{u^{k+1}}{(1+u)^{2k-1}}}{1+\frac{u^{k+1}}{(1+u)^{2k-1}}}\\&=\frac{1-u^3}{(1-u)(1+u)^2}\sum_{j\ge1}(-1)^{j-1}\frac{u^{(k+1)j}}{(1+u)^{(2k-1)j}}.
\end{align*}
Extracting coefficients is now done with the Cauchy integral formula:
\begin{align*}
[z^n]&\frac{1-u^3}{(1-u)(1+u)^2}\sum_{j\ge1}(-1)^{j-1}\frac{u^{(k+1)j}}{(1+u)^{(2k-1)j}}\\
&=\frac1{2\pi i}\oint \frac{dz}{z^{n+1}}\frac{1-u^3}{(1-u)(1+u)^2}\sum_{j\ge1}(-1)^{j-1}\frac{u^{(k+1)j}}{(1+u)^{(2k-1)j}}\\
&=\frac1{2\pi i}\oint \frac{du}{u^{n+1}}(1+u)^{2n-3}(1-u^3)\sum_{j\ge1}(-1)^{j-1}\frac{u^{(k+1)j}}{(1+u)^{(2k-1)j}}\\
&=\sum_{j\ge1}(-1)^{j-1}[u^{n-(k+1)j}](1-u^3)(1+u)^{2n-3-(2k-1)j}.
\end{align*}
Using the binomial theorem yields the assertion.
\end{proof}

\section{Functional equations for the constants}

Two of the constants satisfy attractive and non-trivial functional equations. This phenomenon is not uncommon in the analysis of algorithms; we point out the paper \cite{Kirschenhofer-Prodinger-Schoiss:1987:zur-auswer} where it was first observed and the survey \cite{Prodinger:2004:period-oscil} which contains many references to earlier papers. 

The first example is the constant
\begin{align*}
c_0&=\frac 92\sum_{k\ge1}\frac{2^{2k-1}}{(2^{2k-1}+1)^2}=\frac 92 F(\log2), 
\end{align*}
with
\begin{align*}
F(x)&:=\sum_{k\ge1}\frac{e^{(2k-1)x}}{(e^{(2k-1)x}+1)^2}=\sum_{k\ge1}\frac{e^{-(2k-1)x}}{(1+e^{-(2k-1)x})^2}=\sum_{k,j\ge1}(-1)^{j-1}je^{-(2k-1)jx}.
\end{align*}

\begin{proposition}\label{proposition:functional-equation}We have the functional equation
  \begin{equation*}
    F(x)=\frac1{4x}- \frac{\pi^2}{x^2}F\Bigl(\frac{\pi^2}{x}\Bigr).
  \end{equation*}
\end{proposition}
Since $\frac{\pi^2}{\log^22}F\bigl(\frac{\pi^2}{\log2}\bigr)=0.0000134525077\dots$, we have the \emph{near-identity} $F(\log2)\approx \frac1{4\log2}$.

\begin{proof}
We compute the Mellin transform of it \cite{Flajolet-Gourdon-Dumas:1995:mellin}, which exists (at least) in the fundamental strip
$\langle 2,\infty\rangle$:
\begin{align*}
F^*(s)&=\Gamma(s)\sum_{k,j\ge1}(-1)^{j-1}j(2k-1)^{-s}j^{-s}=\Gamma(s)\zeta(s-1)\zeta(s)(1-2^{2-s})(1-2^{-s}).
\end{align*}
The inversion formula for the Mellin transform gives the original function back
(integration is along vertical lines). We shift the line of integration to the
left and collect residues:
\begin{align*}
F(x)&=\frac1{2\pi i}\int\limits_{(\frac52)}\Gamma(s)\zeta(s-1)\zeta(s)(1-2^{2-s})(1-2^{-s})x^{-s}ds\\
&=\frac1{4x}+\frac1{2\pi
  i}\int\limits_{(-\frac52)}\Gamma(s)\zeta(s-1)\zeta(s)h(s)x^{-s}ds
\end{align*}
for
\begin{equation*}
  h(s) = (1-2^{2-s})(1-2^{-s}).
\end{equation*}
In the remainder of this proof, the relation
\begin{equation}\label{eq:h-property}
  h(2-s)=2^{2s-2}h(s)
\end{equation}
will be the only property of $h(s)$ that we will use.

We now use the duplication formula for the Gamma function and the
functional equation for the Riemann zeta function, a substitution $s=2-u$ and
then again a shift of the line of integration.
\begin{multline*}
\frac1{2\pi
  i}\int\limits_{(-\frac52)}\frac{2^{s-1}}{\sqrt{\pi}}\Gamma\Bigl(\frac{s}{2}\Bigr)\Gamma\Bigl(\frac{s+1}{2}\Bigr)\zeta(s-1)\zeta(s)h(s)x^{-s}ds\\
\begin{aligned}
&=\frac1{2\pi i}\int\limits_{(-\frac52)}\frac{2^{s-1}}{\sqrt{\pi}}\frac{s-1}{2}h(s)
\pi^{2s-2}\Gamma\Bigl(\frac{1-s}{2}\Bigr)\zeta(1-s)\Gamma\Bigl(\frac{2-s}{2}\Bigr)\zeta(2-s)x^{-s}ds\\
&=\frac1{2\pi i}\int\limits_{(\frac92)}\frac{2^{1-u}}{\sqrt{\pi}}\frac{1-u}{2}h(2-u)
\pi^{2-2u}\Gamma\Bigl(\frac{u-1}{2}\Bigr)\zeta(u-1)\Gamma\Bigl(\frac{u}{2}\Bigr)\zeta(u)x^{u-2}du\\
&=-\frac1{2\pi i}\int\limits_{(\frac52)}2^{u-1}h(u)
\pi^{2-2u} \frac{1}{2^{u-1}}  \Gamma(u)\zeta(u-1)\zeta(u)x^{u-2}du\\
&=-\frac{\pi^2}{x^2}\frac1{2\pi i}\int\limits_{(\frac52)}h(u)
\pi^{-2u}  \Gamma(u)\zeta(u-1)\zeta(u)x^{u}du\\
&=-\frac{\pi^2}{x^2}F\Bigl(\frac{\pi^2}{x}\Bigr).
\end{aligned}
\end{multline*}
\end{proof}

Our second example relates to a sum that appears within the constant $d_2$:
\begin{equation*}
S=\frac32\sum_{k\ge1}\frac{2k-1}{2^{2k-1}+1}=\frac32\sum_{k\ge1}\frac{(2k-1)2^{-2k+1}}
{1+2^{-2k+1}}=\frac32\sum_{k,j\ge1}(-1)^{j-1}(2k-1)2^{-j(2k-1)}=
\frac32G(\log 2)
\end{equation*}
with
\begin{equation*}
G(x)=\sum_{k,j\ge1}(-1)^{j-1}(2k-1)e^{-j(2k-1)x}.
\end{equation*}

\begin{proposition}
   We have the functional equation
\begin{align*}
G(x)=\frac{\pi^2}{24x^2}+\frac{1}{24}-\frac{\pi^2}{x^2}G\Bigl(\frac{\pi^2}{x}\Bigr).
\end{align*}
\end{proposition}
Since
$\frac{\pi^2}{\log^22}G\bigl(\frac{\pi^2}{\log2}\bigr)=0.0000134525165276\dots$, we have the near-identity
$G(\log2)\approx \frac{\pi^2}{24\log^22}+\frac{1}{24}$.

\begin{proof}
  The Mellin transform of $G$ is
\begin{align*}
\Gamma(s)\sum_{k,j\ge1}(-1)^{j-1}j^{-s}(2k-1)^{-s+1}
=\Gamma(s)\zeta(s)\zeta(s-1)(1-2^{1-s})^2.
\end{align*}
The proof of Proposition~\ref{proposition:functional-equation} applies with $h(s)$ replaced by $(1-2^{1-s})^2$ which again has
the property~\eqref{eq:h-property}.
\end{proof}

\bibliography{bib/cheub}
\bibliographystyle{amsplainurl}

\end{document}